\documentclass[12pt]{amsart}
\usepackage{graphicx}
\usepackage{amsthm}
\usepackage{amssymb}
\usepackage{amsmath,amsfonts,amstext,amssymb,amsbsy,amsopn,amsthm,eucal}
\usepackage{color}
\usepackage{pstricks}

\usepackage{fullpage}

\newtheorem{tm}{Theorem}[section]
\newtheorem{rem}[tm]{Remark}
\newtheorem{lem}[tm]{Lemma}

\newtheorem{ko}[tm]{Corollary}
\theoremstyle{definition}
\newtheorem{ex}[tm]{Example}

\newcommand{\lmat}{\left[\begin{array}}
\newcommand{\rmat}{\end{array}\right]}
\newcommand{\diag}{{\rm diag\,}}
\renewcommand{\hat}{\widehat}

\def\uinl{\left|\kern-1.5truept\left|\kern-1.5truept\left|}
\def\uinr{\right|\kern-1.5truept\right|\kern-1.5truept\right|}
\def\uinlr{\hbox{$|\kern-1.5truept|\kern-1.5truept|\cdot
                 |\kern-1.5truept|\kern-1.5truept|$}}

\def\wtd{\widetilde}
\def\what{\widehat}

\def\be{\begin{enumerate}}
\def\ee{\end{enumerate}}
\def\R{\mathbb{R}}
\def\C{\mathbb{C}}

\def\cF{\mathcal{F}}
\def\cG{\mathcal{G}}
\def\cal{\mathcal}

\date{}
\DeclareMathOperator{\Ran}{Ran} \DeclareMathOperator{\Ker}{Ker}

\DeclareMathOperator{\RelGap}{RelGap}

\begin{document}
\title{The Rotation of Eigenspaces of Perturbed Matrix Pairs}


\author[LG]{Luka Grubi\v{s}i\'{c}}
\author[NT]{Ninoslav Truhar}
\author[KV]{Kre\v{s}imir Veseli\'{c}}
\address{University of Zagreb, Depertment of Mathematics, Bijeni\v{c}ka 30, 10000 Zagreb, Croatia.}
\email{luka.grubisic@math.hr}
\address{Department of Mathematics,
  University J.J. Strossmayer, Trg Ljudevita Gaja 6, 31000 Osijek, Croatia.}
  \email{ntruhar@mathos.hr}
\address{
Fernuniversit\"at in Hagen, Lehrgebiet Mathematische Physik, 58084 Hagen, Germany.}
\email{kresimir.veselic@fernuni-hagen.de}

\begin{abstract}
We revisit the relative perturbation theory for invariant subspaces of
positive definite matrix pairs. As a prototype
model problem for our results we consider parameter dependent families of eigenvalue problems. We show
that new estimates are a natural way to obtain sharp --- as functions of the parameter indexing the family of
matrix pairs --- estimates for the rotation of spectral subspaces.
\end{abstract}
\maketitle

\section{Introduction and motivation}
This paper is concerned with the notion of the optimality of bounds on the rotation of spectral subspaces
of positive definite Hermitian matrix pairs under the influence of additive perturbations. Precisely, given positive definite
Hermitian matrix pairs $(H, M)$ and $(\widetilde{H}, \widetilde{M})=(H+\delta H, M+\delta M)$ and their spectral subspaces $\mathcal{E}$ and $\tilde{\mathcal{E}}$ of the
same dimensionality we provide estimates
\begin{equation}\label{eq:I1}
\|\sin\Theta_M(\mathcal{E},\tilde{\mathcal{E}})\|\leq\text{Gap}_1\frac{\eta_H}{\sqrt{1-\eta_H}}+
\text{Gap}_2\frac{\eta_M}{\sqrt{1-\eta_M}}
\end{equation}
where $\eta_A=\|A^{-1/2}(A-\tilde A)A^{-1/2}\|$ is the usual relative distance between positive definite Hermitian matrices
$A$ and $\tilde A$, $\text{Gap}_i$ measure the gaps in the spectrum and $\|\sin\Theta_M(\mathcal{E},\tilde{\mathcal{E}})\|$
measures the size of the rotation in the scalar product $(x,y)_M=x^*My$ dependent on the matrix $M$.
For more on $\sin\Theta$ theorems see \cite{Davis1970,Grubivsi'c2007,Li1999a,Li1999,Stewart1990}. In
comparison, we approach the problem of the changing scalar product by presenting our estimates in the $M$-scalar product,
whereas the standard approach yields estimates in the Euclidean scalar product.

Let us now consider the notion of the optimality of perturbation estimates in the context of parameter dependent
perturbation families.
In this setting we analyze rotations of eigenspaces of positive definite Hermitian matrix pairs $(H,M)$ under the influence
of a parameter dependent family of perturbations. The allowed families of perturbations $\delta H_\kappa$ and $\delta M_\kappa$ ---
where $\kappa$ is some indexing parameter --- are assumed to satisfy the restrictions
\begin{align}\label{assup:1}
\big|x^*\delta H_\kappa y\big|\leq \cF(\kappa)\sqrt{x^*Hx\;y^* Hy}&&\big|x^*\delta M_\kappa y\big|\leq \cG(\kappa)\sqrt{x^*Mx\;y^*My}\\
\lim_{\kappa\to\infty} \cF(\kappa)=0&&\lim_{\kappa\to\infty} \cG(\kappa)=0.
\end{align}
Here the matrix valued functions $\delta H_\kappa$ and $\delta M_\kappa$ are assumed to take value in the space of Hermitian matrices of appropriate size,
and by a convention $x^*$ denotes the transpose or Hermitian transpose of an object $x$ --- be it matrix or vector ---
as is given by the context. We also assume that  $\cF$ and $\cG$ are some real valued functions and we apply (\ref{eq:I1})
by setting $H_\kappa:=H+\delta H_\kappa$ and $M_\kappa:=M+\delta M_\kappa$ and noting the estimates $\eta_{H_\kappa}\leq\cF(\kappa)$ and $\eta_{M_\kappa}\leq\cG(\kappa)$
if we set $\widetilde{H}=H_\kappa$ and $\widetilde{M}=M_\kappa$.

It is our aim to argue that matrix dependent scalar product gives a natural
environment to obtain optimal convergence estimates as functions of the parameter $\kappa$. Further feature of our theory is that our estimates are invariant\footnote{The value of $\text{Gap}_i$ does not change under this transformation of the problem.} to the ``inversion of the problem'', so we can also obtain estimates in the $H$ based scalar product by switching the roles of $H$ and $M$. In this context the reader should also note that the identity (\ref{assup:1}), under the assumption that $\kappa$ is such that $\cF(\kappa)<1$ and $\cG(\kappa)<1$ yields the estimates (cf. equation (\ref{cf_eq}))
\begin{align}\label{assup:2}
\big|x^*(H_\kappa^{-1}-H^{-1}) y\big|&\leq \frac{\cF(\kappa)}{1-\cF(\kappa)}\sqrt{x^*H^{-1}x\;y^* H^{-1}y}\\\big|x^*(M_\kappa^{-1}-M^{-1}) y\big|&\leq \frac{\cG(\kappa)}{1-\cG(\kappa)}\sqrt{x^*M^{-1}x\;y^*M^{-1}y}.
\end{align}

Let us now give more insight into the applications which are covered by the assumptions (\ref{assup:1}). This structure is rich enough to
include discretization matrices approximating several singularly perturbed families of problems appearing in mathematical physics. Among other
applications, the penalty methods for Stokes and Maxwell equations from \cite{Warburton2006} can be analyzed in this context, too.
The parameter $\kappa$ is then called the penalty parameter, and it is of interest what happens to the eigenvalues
and eigenspaces as $\kappa\to\infty$. In \cite[Section 4]{Warburton2006} the authors have studied the
perturbation of eigenvalues by a very elegant Gerschgorin type argument and in this paper
we give an eigenspace counterpart of such a result.
For more details see \ref{App} and the explicitly solved academic model problems from Section \ref{Acc1}.

Also, the effect of numerical integration on the rotation of eigenspaces, when assembling finite element mass and stiffness matrices, is covered by (\ref{assup:1}). In this context $\kappa$ is the parameter describing the effect of increasing accuracy of the integration formula. This approach is also used for ``mass lumping'' which amounts to constructing a diagonal matrix $D=M+\delta M$, with $\delta M$ small in some sense. For further information and references see the paper \cite{Banerjee1990}, \ref{App} and the academic example from Section \ref{Acc2} where we rather
favorably compare our results with those that follow from the standard reference \cite{Stewart1990}.

We end this discussion by noting that similar energy norm estimates for eigenvectors have been obtained in \cite{Hetmaniuk2006a} in the context of
the analysis of Laczos method. Furthermore, the authors show how to efficiently compute the ingredients of the estimator in
the context of computationally competitive numerical linear algebra procedures. We extend some of those results by giving a subspace version of some of the estimates, e.g. see appropriate parts of \cite[Proposition 3.3 and 3.4]{Hetmaniuk2006a} and compare with our numerical results from Section \ref{num}. It is possible that
our subspace results could be of technical help when developing a similar analysis of the block Lanczos method.

We now turn to the main question of this paper. What is the real nature of the sharpness claim of a $\sin\Theta$ theorem?
Many of such theorems are obtained under essentially different spectral assumptions. Each is claimed to be sharp by constructing
an appropriate example where the bound is attained. The results cannot be readily compared, even though
one class of results can be seen to be following
from the other, since their optimality depends on the set of assumption which were necessary
to obtain the results. Quantitatively,
transforming one class of results into the other type of estimates changes the quantitative
performance of the results so considerably that a direct comparison
is no longer fair. In a sense, each result is ``sharp'' given the setting in which it has
been obtained so discussion is
more about which set of assumptions are more appropriate than the others.

We do not further address this fundamental questions. Instead, we opt to normalize the estimates by dividing the measure of the rotation which is being estimated
with the estimator and then compare various estimates on specially tailored model problems.
A first logical candidate --- in a single matrix case --- for a competing estimate would be a $\sin\Theta$ theorem from
\cite{Davis1970,Li1999,Grubivsi'c2007}. However,
it turns out that this estimate --- as the function of $\kappa$ --- is overly pessimistic.
Our aim is in particular
to derive sharp estimates for the rotation of eigenspaces for this class of
problems given by the parameter dependent family $H_\kappa=H+\delta H_\kappa$. The solution is to look
for the rotation of eigenspaces in the energy norm, that is $H$ based.
In our setting this boils down to the analysis of the matrix pair
$(H_\kappa^{-1}, H_\kappa)$. Let us also point
out that we will discuss sharpness, or lack of it, in the various $\sin\Theta$
results by comparing
the residual type estimates which can be obtained for matrix pairs
$$
(H_\kappa,I),\;(I,H_\kappa),\;(H_\kappa^{-1},H_\kappa),\;(H_\kappa^{-1},I),\;(I,H_\kappa^{-1}).
$$
We will conclude that any estimate of an eigenspace rotation under the assumptions (\ref{assup:1}) is actually meant to be
in a matrix dependent scalar product and that it will under-perform if used to measure rotations in the Euclidean scalar product.

\section{Notations, definitions and the general setting}
The optimal setting to consider all of the above eigenvector problems is an analysis of
the whole class of positive definite matrix pairs $(H, M)$,
where $H$ and $M$ are positive definite. More to the point,
we consider the following generalized eigenvector problem
\begin{eqnarray}\label{eigprobl1}
H x & = & \lambda M x ,
\end{eqnarray}
and the corresponding perturbed one
\begin{eqnarray}\label{eigprobl1t}
(H+\delta H) \wtd x & = & \wtd \lambda (M + \delta M) \wtd x \,,
\end{eqnarray}
where  $H, M$, $\wtd H \equiv H+\delta H$,
$ \wtd M \equiv M + \delta M  \in \C^{n \times n}$ are Hermitian positive definite.

\subsection{Spectral theorem and block operator matrix notation}Under these assumptions  matrix pairs $(H, M)$ can be simultaneously diagonalized, that is
there exists a non-singular matrix $X$ such that
\begin{align} \label{diagpar1}
X^*
H X = \Lambda, \quad X^* M X = I,
\end{align}
where
$
\Lambda = \diag(\lambda_1, \ldots, \lambda_n)
\, \quad \lambda_i \in \R $
for $i = 1, \ldots, n$ and we use $X^*$ to denote the Hermitian adjoint.

We will represent our perturbation problem by block operator matrices and will use the following
notation for the perturbation problems which will be needed in the analysis. Let us decompose  $X$ and $\wtd X $ as
\begin{eqnarray*} 
X = \begin{bmatrix}X_{1} & X_{2} \end{bmatrix} & & \wtd X = \begin{bmatrix}\wtd X_{1} &
\wtd X_{2} \end{bmatrix},
\end{eqnarray*}
where $X_{1}, \wtd X_{1} \in \C^{n \times k}$ and
$X_{2}, \wtd X_{2} \in \C^{n \times n-k}$.
The eigen-decomposition (\ref{diagpar1}) can now be written as
\begin{eqnarray} \label{diagvec1}
\begin{bmatrix}X^*_{1} \\ X^*_{2} \end{bmatrix} H \begin{bmatrix}X_{1} & X_{2} \end{bmatrix}
= \begin{bmatrix}\Lambda_{1} & 0  \\ 0 & \Lambda_{2} \end{bmatrix} , &&
\begin{bmatrix}X^*_{1} \\ X^*_{2} \end{bmatrix} M \begin{bmatrix}X_{1} & X_{2} \end{bmatrix}
= \begin{bmatrix}I_{k} & 0 \\ 0 & I_{n-k} \end{bmatrix} .
\end{eqnarray}
Similarly as above, for perturbed quantities one can write
\begin{eqnarray} \label{diagvec1t} \begin{bmatrix}\wtd
X^*_{1} \\ \wtd X^*_{2}  \end{bmatrix} \wtd H
\begin{bmatrix}\wtd X_{1} & \wtd X_{2} \end{bmatrix} =
\begin{bmatrix} \wtd
\Lambda_{1} & 0 \\ 0 & \wtd \Lambda_{2} \end{bmatrix} , &  &
\begin{bmatrix}\wtd X^*_{1} \\ \wtd X^*_{2} \end{bmatrix} \wtd M
\begin{bmatrix}\wtd X_{1} & \wtd X_{2}  \end{bmatrix}   =
\begin{bmatrix} I_{k} & 0 \\ 0 & I_{n-k}  \end{bmatrix}.\label{druga}
 \end{eqnarray}
and
\begin{eqnarray} \label{diagvec2t} \begin{bmatrix}\what
X^*_{1} \\ \what X^*_{2}  \end{bmatrix} \wtd H
\begin{bmatrix}\what X_{1} & \what X_{2} \end{bmatrix} =
\begin{bmatrix} \what
\Lambda_{1} & 0 \\ 0 & \what \Lambda_{2} \end{bmatrix} , &  &
\begin{bmatrix}\what X^*_{1} \\ \what X^*_{2} \end{bmatrix} M
\begin{bmatrix}\what X_{1} & \what X_{2}  \end{bmatrix}   =
\begin{bmatrix} I_{k} & 0 \\ 0 & I_{n-k}  \end{bmatrix},
 \end{eqnarray}
where $\what X_{1} \in \C^{n \times k}$ and
$\what X_{2} \in \C^{n \times n-k}$ and $\what X = \begin{bmatrix}\what X_{1} &
\what X_{2} \end{bmatrix}$.

\subsection{Measuring perturbations of positive definite matrices}
The size of the perturbations $\delta H$ and $\delta M$ will be measured in the relative sense. This
means that we assume that we have information on the singular values of the matrices
\begin{equation}\label{eq:1}
H^{-1/2}(H-\wtd H){\wtd H}^{-1/2}\;\text{and}\;M^{-1/2}(M-\wtd M){\wtd M}^{-1/2}
\end{equation}
or
\begin{equation}\label{eq:2}
H^{-1/2}(H-\wtd H){H}^{-1/2}\;\text{and}\;M^{-1/2}(M-\wtd M){M}^{-1/2}.
\end{equation}
Typically we only use the maximal singular value, that is the spectral norm estimate
of these relative perturbations. More to the point we will use the quantities defined in the lemma
below in most of our arguments. In this paper we use $\|\cdot\|_2$ to denote the spectral matrix norm, and
$\|\cdot\|$ to denote any unitary invariant matrix norm, when there is no danger of confusion.
\begin{lem}\label{lem1}
Let $H$ be a positive definite matrix and let $\Psi_H=\|H^{-1/2}(H-\wtd H){\wtd H}^{-1/2}\|$
and $\eta_H=\|H^{-1/2}(H-\wtd H){H}^{-1/2}\|_2$. Then for any $x, y\in\C^n$
\begin{align}
|x^*(H-\wtd H)x|&\leq\eta_H~ x^*Hx,\\
|x^*(H-\wtd H)y|&\leq\frac{\eta_H}{\sqrt{1-\eta_H}}~ \sqrt{x^*Hx~x^*{\wtd H}x},\\
\|H^{-1/2}(H-\wtd H){\wtd H}^{-1/2}\|&\leq\frac{1}{\sqrt{1-\eta_H}}
\|H^{-1/2}(H-\wtd H){H}^{-1/2}\|.\label{zadnja}
\end{align}
In particular, relation (\ref{zadnja}) reduces to $\Psi_H\leq\frac{\eta_H}{\sqrt{1-\eta_H}}$ in
the case $\|\cdot\|=\|\cdot\|_2$.
\end{lem}
The
proof is by direct computation, see also \cite{Grubivsi'c2007}. Let us note that our theory is not
limited to the use of spectral norm only. We allow for the consideration of any unitary invariant norm of
the perturbations (\ref{eq:1}) and (\ref{eq:2}).
\begin{rem}
When there is a danger of confusion we will use the notation
\begin{equation}\label{Psi_Fro}
\Psi_H^{\|\cdot\|}=\|H^{-1/2}(H-\wtd H){\wtd H}^{-1/2}\|
\end{equation}
to denote the dependence of the perturbation measure on the unitary invariant norm.
\end{rem}
\begin{rem}\label{L2lumping}
Let us note that in an application of this theory in the setting of the mass lumping
finite element methods we consider the perturbations of the $M$ matrix. The constant $\eta_M$ for such a
perturbation typically depends on mesh parameters. Furthermore, let us note that if there exist
constants $\delta_1$, $\delta_0$, $0<\delta_0\leq\delta_1$ such that
$$
\delta_0~x^*  D x\leq x^* M x\leq\delta_1~x^* D x
$$
holds for some symmetric positive definite matrices $D$ and $M$, then
$
\wtd M=\frac{\delta_1+\delta_0}{2}D
$
has the property
$$
\frac{2\delta_0}{\delta_0+\delta_1}x^*  \wtd M x\leq x^* M x\leq\frac{2\delta_1}{\delta_0+\delta_1} x^* \wtd M x
$$
which can be written as
$$|x^*(M-\wtd M)x|\leq \frac{\delta_1-\delta_0}{\delta_1+\delta_0}~ x^*\wtd Mx~.
$$
\end{rem}

\subsection{Relations between subspaces in the changing scalar product}
Let us now define the basic tools which will be used to compare subspaces of $\C^n$. Let $\mathcal{X}$
and $\mathcal{Y}$ be some generic $m$-dimensional subspaces of $\C^n$. For any of such subspaces
there are bases\footnote{By saying the basis $Y$ we mean ``the basis
           given by the columns of $Y$''.} $X, Y\in\C^{n\times m}$ such that
           $\mathcal{X}=\Ran(X)$ and $\mathcal{Y}=\Ran(Y)$.
 Let us choose $X$ and $Y$ such that $X^*X=Y^*Y=I_{m}$ then $P_{\mathcal{X}}=XX^*$ and $P_{\mathcal{Y}}=YY^*$
 are orthogonal projections onto $\mathcal{X}$ and $\mathcal{Y}$. Typically, one compares the subspaces
 $\mathcal{X}$ and $\mathcal{Y}$ by analyzing the spectral properties of the product
 $S_{(\mathcal{X}, \mathcal{Y})}=(I-P_{\mathcal{X}})P_{\mathcal{Y}}$.
The $m$-singular values of $S_{\mathcal{X}, \mathcal{Y}}\Big|_\mathcal{X}$---the restriction of
$S_{\mathcal{X}, \mathcal{Y}}$ on $\mathcal{X}$---are called the sines of the angle
between the subspaces $\mathcal{X}$ and $\mathcal{Y}$. In the matrix notation they
are exactly the $m$-singular values of the matrix
$$
S_{\mathcal{X}, \mathcal{Y}}=(I-XX^*)Y.
$$
This is the measure of the size of the rotation\footnote{Such rotation exists if all of the sines of the angle between $\mathcal{X}$ and $\mathcal{Y}$ are strictly smaller than one.} in $\C^n$ which would
move the subspace $\mathcal{X}$ onto $\mathcal{Y}$.

In this note we analyze the angles between the subspaces $\mathcal{X}$ and $\mathcal{Y}$ in the scalar
product $(x, y)_M=x^*My$, $x, y\in\C^n$ which is defined by the positive definite matrix $M$. To this end let $X^*MX=Y^*MY=I_{m}$, which is to say let $X$ and $Y$ be $M$-unitary. Then the sines of the angle between
$\mathcal{X}$ and $\mathcal{Y}$ in the $M$-scalar product are the $m$-singular values of the matrix product
$$
S^M_{\mathcal{X}, \mathcal{Y}}=M^{1/2}(I-XX^*M)Y
$$
For more on angles between the subspaces of $\C^n$ see \cite{Davis1970,Knyazev2002}.

Let us not that since both \(M\) and \(H\) are subject
to perturbation particular care is needed because the underlying space
geometry changes with \(M\). We shall therefore simplify the subsequent discussion of the $M$-product dependent
subspace angles.

In order to be definite we shall concentrate---and give explicit formulae for the angles---only on the relationship between the subspaces of interest for our analysis.
That is we consider the relationship between the subspaces $\Ran(X_{1})$, $\Ran({\wtd X}_{1})$ and
$\Ran({\what X}_{1})$.

The columns of $X_1$ and ${\what X_1}$ are $M$-orthogonal, then we use
the following characterization of the sines of the canonical angles
between the $M$ orthogonal subspaces ${\cal X}_{1} = \Ran(X_{1})$ and
${\what {\cal X}}_{1} = \Ran({\what X}_{1})$
induced by weighted $M$-inner product:
\begin{align}\label{def:SinTheta-M}
\sin{\Theta_M({\cal X}_{1}, {\what {\cal X}}_{1} )}   =
 {\what X}^*_2 M X_1 \,.
 \end{align}
Let us now consider the problem of the changing scalar product.
Since $\wtd M = M + \delta M$,
it follows from (\ref{druga}) that
\begin{align*}
\wtd X^* M \wtd X = I - \wtd X^* \delta M \wtd X\,.
\end{align*}
Assume that $I - \wtd X^* \delta M \wtd X$ is positive definite. Then
 $\what X$ and $\wtd X Y^{-*} $ are $M$-orthogonal, where $Y$ is Cholesky factor
 such that $Y Y^* =I - \wtd X^* \delta M \wtd X $.
%
%
%
We can now use a similar characterization of the sines of canonical angles as in (\ref{def:SinTheta-M}).
We show that the  sines of the canonical angles
between the eigenspaces $\what{\cal X}_{1} = \Ran(\what X_{1})$ and
$\wtd{\cal X}_{1} = \Ran(\wtd X_{1})$
induced by weighted $M$-inner product are given by:
\begin{align}\label{def:SinTheta-M2}
\sin{\Theta_M(\what {\cal X}_{1}, \wtd {\cal X}_{1} )}   =
\what X^*_2 M \wtd X_1 Y_{11}^{-*}, \quad {{\mbox{ where } }}  \quad
Y = \begin{bmatrix}Y_{11} & Y_{12} \\ 0 & Y_{22} \end{bmatrix} \,.
 \end{align}

Finally, let $P_{\cal X} = X X^*$ be orthogonal projector
onto $m$-dimensional subspace $\mathcal{X} = \Ran(X)$, where $X$ satisfies
$X^*X = I_m$. Using the  \cite[Theorem II 4.10.]{Stewart1990},
one can write
\begin{align} \label{neq:NormSinTheta_pom}
\left\|\sin{\Theta({\cal X}, {\cal Y})} \right\|  =
\left\| P_{\cal X} - P_{\cal Y} \right\| =
\left\| (I - P_{\cal X}) P_{\cal Y} \right\| = \left\| (I - P_{\cal Y}) P_{\cal X} \right\|\,,
 \end{align}
for any unitary invariant norm $\|\cdot\|$.
Further, note that the columns of the matrices $X_1^M = M^{1/2} X_1$, $\what X_1^M = M^{1/2} \what X_1$
and $\wtd X_1^M = M^{1/2} \wtd X_1 Y_{11}^{-*}$ are unitary, thus using (\ref{neq:NormSinTheta_pom}),
one can write
\begin{align}\label{neq:NormSinTheta_pomB}
\left\|\sin{\Theta_M( {\cal X}_{1}, \what {\cal X}_{1} )} \right\|  =
\left\| P_{{\cal X}_{1}} - P_{\what {\cal X}_{1}} \right\| =
\left\| \what X_2^* M X_1 \right\|\,.
 \end{align}
where $P_{{\cal X}_{1}} = X_1^M (X_1^M)^*$ and where $P_{\what {\cal X}_{1}} = \what X_1^M (\what X_1^M)^*$,
and similarly
\begin{align*} 
\big\|\sin{\Theta_M( \what {\cal X}_{1}, \wtd {\cal X}_{1} )} \big\|  =
\big\| P_{\what {\cal X}_{1}} - P_{\wtd {\cal X}_{1}}  \big\| =
\big\| \what X_2^* M \wtd X_1 Y_{11}^{-*} \big\|\,,
 \end{align*}
where $P_{\what {\cal X}_{1}} = \what X_1^M (\what X_1^M)^*$ and
$P_{\wtd {\cal X}_{1}} = \wtd X_1^M (\wtd X_1^M)^*$.

The results above imply the upper bound for the sines of the canonical angles
between the eigenspaces  ${\cal X}_{1} = {\cal
R}(X_{1})$ and $\wtd {\cal X}_{1} = \Ran(\wtd X_{1})$, can be obtained in any unitary invariant norm $\|\cdot\|$ ,
using the simple triangle inequality. We have
\begin{align} \label{neq:NormSinTheta}
 \big\|\sin{\Theta_M({\cal X}_{1}, \wtd {\cal X}_{1} )} \big\|  \leq
\big\| \sin{\Theta_M( {\cal X}_{1}, \what {\cal X}_{1} )} \big\| +
\big\| \sin{\Theta_M(\what {\cal X}_{1}, \wtd {\cal X}_{1} )} \big\|\,,
 \end{align}
that is $ \big\|\sin{\Theta_M({\cal X}_{1}, \wtd {\cal X}_{1} )} \big\|$
can be estimated as the sum of the upper bounds for the norms of
the sines matrices from (\ref{def:SinTheta-M}) and (\ref{def:SinTheta-M2}).

\section{The main result}

Our aim is to derive a bound for the sines of the canonical angles between eigenspaces
${\cal X}_{1} = \Ran(X_{1})$ and $\wtd {\cal X}_{1} = \Ran(\wtd X_{1})$ from (\ref{diagvec1}) and (\ref{diagvec1t}).

This will be done with the two steps procedure as suggested by the form of the inequality (\ref{neq:NormSinTheta}).
This approach is in the line with the pioneering analysis of the relative sensitivity of the eigenvalues of a positive-definite matrix pair from
\cite{Barlow1990}.

The road-map for the prof is outlined in the following list. For each preparatory step we will prove a theorem to
justify the procedure and at the end we will combine the conclusion in the main theorem. The
preparatory steps can be classified as follows:
\begin{enumerate}
\item \(H\) perturbed, \(M\) unchanged
\begin{align}\label{eq:PertI}
X^* H X = \Lambda, \quad X^* M X = I, \qquad  \what X^* \wtd H \what X = \what \Lambda,
\quad \what X^* M \what X =I,
\end{align}
where
$\Lambda = \diag(\lambda_1, \ldots, \lambda_n)\,,$
$ \what \Lambda = \diag(\what \lambda_1, \ldots, \what \lambda_n)\,,$
$  \lambda_i, \what \lambda_i \in \R$, for $i = 1, \ldots, n$.
\item \(M\) perturbed, \(H\) unchanged
\begin{align} \label{PertStep2}
X^* \wtd  H  X = \what  \Lambda, \quad
X^* M \what  X = I, \qquad
\wtd X^* \wtd H \wtd X = \wtd \Lambda, \quad \wtd X^* \wtd M \wtd X= I,
\end{align}
where
$ \Lambda = \diag( \lambda_1, \ldots, \lambda_n)\,,$ and
$\wtd \Lambda = \diag(\wtd \lambda_1, \ldots, \wtd \lambda_n)\,,$
and $\lambda_i,  \wtd \lambda_i \in \R$, for $i = 1, \ldots, n$.
\end{enumerate}

The main tools in our analysis will be sharp estimates for the solution of the structured Sylvester
equations from \cite[Lemma 2.4]{Li1999} and \cite[Lemma 2.3]{Li1999}. That is, we consider the
structured Sylvester equations\footnote{The solution of (\ref{eq:S1}) is presented in \cite[Lemma 2.4]{Li1999}. This equation has also been analyzed in infinite dimensional setting in \cite{Grubivsi'c2007}. The equation (\ref{eq:S2})
has been analyzed in \cite[Lemma 2.3]{Li1999}, see also \cite{Li1999a}.}
\begin{align}\label{eq:S1}
AX-XB&=A^{1/2}CB^{1/2}\\
AX-XB&=CB.\label{eq:S2}
\end{align}

\subsection{The first step}
Now we will state our first theorem. We will use the notation and the conclusions of
Lemma \ref{lem1} without further comments.

\begin{tm} \label{gptmsubsp2}
Let $(H, M)$ be a Hermitian pair defined by (\ref{eigprobl1}) and
let $(\wtd H, M)$ be perturbed pair defined by
\begin{eqnarray*} 
(H+\delta H) \what x & = & \what \lambda M \what x \,.
\end{eqnarray*}
Let $X = \begin{bmatrix}X_{1} & X_{2} \end{bmatrix}$ and
$\what X = \begin{bmatrix}\what X_{1} & \what X_{2} \end{bmatrix}$, be
non-singular matrices  which simultaneously diagonalize the pairs
$(H, M)$ and $(\wtd H,  M)$, as in (\ref{eq:PertI}). By setting $\Psi_H=\|H^{-1/2}(H-\wtd H){\wtd H}^{-1/2}\|$ we have
\begin{align}\label{ocjsinpar2}
\|\sin{\Theta_M({\cal X}_{1}, \what {\cal X}_{1} )} \| & \leq
\frac{\Psi_H}{\RelGap} \,,
\end{align}
where
\begin{eqnarray} \label{relgap21}
\RelGap = \min_{\stackrel{\lambda_i \in \Lambda_2 }{ \what \lambda_j \in \what\Lambda_1}}
\frac{|\lambda_i - \what \lambda_j|}{\sqrt{|\lambda_i| |\what \lambda_j|}}
& &
\Lambda_2 = \diag(\lambda_{k+1}, \ldots, \lambda_n ), \quad \what\Lambda_1 =
\diag(\what \lambda_1, \ldots, \what \lambda_k )\,.
\end{eqnarray}

\end{tm}
\begin{proof}
Since, according to (\ref{def:SinTheta-M}),
$$\sin{\Theta_M({\cal X}_{1}, \what {\cal X}_{1} )}   =
 X^*_2 M \what X_1\,,$$
we  have to bound $\| X^*_2 M \what X_{1} \|$.
By the definition we have
$X^* H X = \Lambda$, and so one can write
\begin{eqnarray}\label{itaproj}
H^{1/2} X = U \Lambda^{1/2}  \,, 
\end{eqnarray}
where $U=\begin{bmatrix}U_1 &U_2\end{bmatrix}=H^{1/2}X\Lambda^{-1/2}$ is unitary and has the block structure
conforming to the structure of $X$. A similar identity also holds for perturbed quantities.
On the other hand, for perturbed quantities it also holds 
$$ (H + \delta H) \what X_1 = M
\what X_1 \what \Lambda_1, $$ where $\what \Lambda_1 = \diag(\what \lambda_1, \ldots, \what \lambda_r)$,
and similarly for unperturbed quantities. We multiply the
above equality by $X_2^*$ from the left, and get
 \begin{align*} 
X_2^* H  \what X_1 -  X_2^* M  \what X_1 \what \Lambda_1 =  - X_2^* \delta H  \what X_1~.
\end{align*}
Using the fact that $H X_2 = M X_2 \Lambda_2 $, this identity can be transformed into
 \begin{align} \label{eq:SylvPertI}
\Lambda_2 X_2^* M  \what X_1 -  X_2^* M  \what X_1 \what \Lambda_1 =  - X_2^* \delta H  \what X_1 \,.
\end{align}
We will proceed by  rearranging the right-hand side of (\ref{eq:SylvPertI}). For that purpose note that one can rewrite the right-hand side of (\ref{eq:SylvPertI}) as
 \begin{align} \label{eq:SylvPertI-2}
 X_2^* \delta H  \what X_1  =  X_2^* H^{1/2} H^{-1/2} \delta H  {\wtd H}^{-1/2}{\wtd H}^{1/2} \what X_1\,,
\end{align}
which together with (\ref{itaproj}) gives
 \begin{align*} 
 X_2^* \delta H  \what X_1  = \Lambda_2^{1/2} U_2^* H^{-1/2} \delta H   {\wtd H}^{-1/2}  \what U_1 \what \Lambda_1^{1/2}\,.
\end{align*}

The above equality and (\ref{eq:SylvPertI}) give
 \begin{align} \label{eq:SylvPertII}
\Lambda_2 X_2^* M  \what X_1 -  X_2^* M  \what X_1 \what \Lambda_1 =  -\Lambda_2^{1/2} U_2^* H^{-1/2} \delta H   {\wtd H}^{-1/2}  \what U_1 \what \Lambda_1^{1/2}\,.
\end{align}
This identity can be recognized as the structured Sylvester equation from (\ref{eq:S1}). This equation
is even meaningful when $H$ and $M$ are unbounded operators. In this setting it is called the weak
Sylvester equation and it has been analyzed in \cite{Grubivsi'c2007}.

Applying  \cite[Lemma 2.4]{Li1999} to obtain the bounds on the solution of the structured Sylvester
  equation (see also \cite{Li1999a}), on (\ref{eq:SylvPertII}) one gets, see (\ref{Psi_Fro}):
\begin{align}
\| X_2^* M  \what X_1 \| \leq \frac{\Psi_H^{\|\cdot\|}}{\RelGap }\,,
\qquad {\mbox{ where }} \qquad
\RelGap  = \min_{\lambda_i \in  \Lambda_2 \,, \what \lambda_j \in \what \Lambda_1} \frac{\displaystyle{| \lambda_i -  \what \lambda_j|}}{\displaystyle{\sqrt{\lambda_i \, \what \lambda_j}}}\,,
\label{eq:Bound1}
\end{align}
for any unitary invariant norm $\|\cdot\|$. \qquad   \end{proof}

\subsection{The second step---the change in scalar product}

Here we will derive the upper bound for the sines of the canonical angles between the eigenspaces $\what{\cal X}_{1} = \Ran(\what X_{1})$ and
$\wtd{\cal X}_{1} = \Ran(\wtd X_{1})$
induced by weighted $M$-inner product, defined by
\begin{align}\label{def:SinTheta-M2b}
\sin{\Theta_M(\what {\cal X}_{1}, \wtd {\cal X}_{1} )}   =
\what X^*_2 M \wtd X_1 Y_{11}^{-*}\,,
 \end{align}
where  \begin{align}\label{def:CholFact}
 \begin{bmatrix}Y_{11} & 0 \\ Y_{21} & Y_{22} \end{bmatrix}
  \begin{bmatrix}Y_{11}^* & Y_{21}^* \\ 0 & Y_{22}^* \end{bmatrix}
  = I - \wtd X^* \delta M \wtd X \,.
 \end{align}

\begin{rem}  Note that one of  possibilities to chose $Y_{11}$ in (\ref{def:CholFact})
can be obtained by Block Cholesky elimination applied on the right-hand side in (\ref{def:CholFact}). This choice yields the block
$Y_{11}= \sqrt{I -  \wtd X_1^* \delta M  \wtd X_1}$.
\end{rem}

We now consider the problem of the perturbation of the matrix pair
$(\wtd  H, M)$ to $(\wtd  H, \wtd M)$. The following theorem contains the upper bound for the $\|\what X^*_2 M \wtd X_1\|$, where
$\| \cdot\|$ stands for any unitary invariant  norm.

\begin{figure}[ht]
\begin{pspicture}(12,2)
\psline(0.5,1)(12,1)
\psline[linewidth=3pt](8,1)(12,1)
\psline[linewidth=3pt](6,1)(7,1)
\rput(6.5,1.5){${\wtd \Lambda}_1$}
\rput(9,1.5){${\what \Lambda}_2$}
\rput(7,.5){$\alpha$}
\rput(8.2,.5){$\alpha+\delta$}
\rput(5,.5){$0$}
\psdot*[dotscale=1](5,1)
\end{pspicture}
\caption{Spectral configuration for Theorem \ref{gptmsubspstep2}.}\label{fig1}
\end{figure}
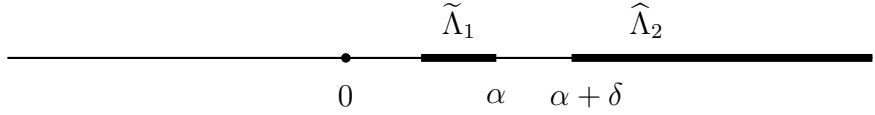

\begin{tm} \label{gptmsubspstep2}
Let $(\wtd H, M)$ be a Hermitian pair  and
let $(\wtd H, \wtd M)$ be perturbed pair defined by
\begin{eqnarray*} 
(H+\delta H) \wtd y & = & \wtd \lambda \wtd M \wtd y \,.
\end{eqnarray*}
Let $\what X = \begin{bmatrix}\what X_{1} & \what X_{2} \end{bmatrix}$ and
$\wtd X = \begin{bmatrix} {\wtd X}_{1} & {\wtd X}_{2} \end{bmatrix}$, be
non-singular matrices  which simultaneously diagonalize the pairs
$(\wtd H, M)$ and $(  \wtd H  , \wtd M)$, as in (\ref{PertStep2}). If
\begin{align} \label{condonSpectA}
& \|\what \Lambda_2\|  \leq \alpha   &{\mbox {\qquad and \qquad}}
& \|\wtd \Lambda_1^{-1}\|^{-1} \geq \alpha + \delta  &  {\mbox {\qquad or \qquad}} \\
 & \|\what \Lambda_2^{-1}\|^{-1}  \geq \alpha + \delta &  {\mbox {\qquad and \qquad}}
& \|\wtd \Lambda_1\| \leq \alpha
\label{condonSpectB}
\end{align}
where $
\what \Lambda_2 = \diag(\what \lambda_{k+1}, \ldots, \what \lambda_n )$,
$ \wtd\Lambda_1 = \diag(\wtd \lambda_1, \ldots, \wtd \lambda_k )$, see Figure \ref{fig1},
then
\begin{align}\label{ocjsinpar2b}
\left\|\what X^*_2 M \wtd X_1 \right\| & \leq
\frac{\Psi_M}{\RelGap_p} \,.
\end{align}
Here we have used $\Psi_M=\|M^{-1/2}(M-\wtd M){\wtd M}^{-1/2}\|$ and for all $1\leq p \leq \infty$ and we have
\begin{eqnarray} \label{relgap21b}
\frac{\delta}{ \alpha + \delta} \geq
\min_{\stackrel{\what \lambda_i \in \what \Lambda_2 }{ \wtd \lambda_j \in \wtd\Lambda_1}}
\frac{\hspace*{-0.25cm}|\what \lambda_i - \wtd \lambda_j|}{\phantom{a}\left( \what \lambda_i^p + \wtd \lambda_j^p \right)^{1/p}}=:\RelGap_p~.
\end{eqnarray}
\end{tm}

\begin{proof}
For the perturbed quantities it holds that
$$ (H + \delta H) \wtd X_1 = \wtd M
\wtd X_1 \wtd \Lambda_1, $$ where $\wtd \Lambda_1 = \diag(\wtd \lambda_1, \ldots, \wtd \lambda_r)$,
and similarly for the unperturbed quantities. Now, by multiplying the
above equality by $\what X_2^*$ from the left, we get
 \begin{align*} 
\what X_2^* \wtd H  \wtd X_1 -  X_2^* \wtd M  \wtd X_1 \wtd \Lambda_1 =  0 \,.
\end{align*}
Using the fact $\wtd H \what X_2 = M \what X_2 \what \Lambda_2 $ (see (\ref{PertStep2}))  this gives
 \begin{align} \label{eq:SylvPertIb}
\what \Lambda_2 \what X_2^* M  \wtd X_1 -  \what X_2^* M  \wtd X_1 \wtd \Lambda_1 =
 - \what X_2^* \delta M  \wtd X_1 \wtd \Lambda_1  \,.
\end{align}

We will proceed by  rearranging the right-hand side of (\ref{eq:SylvPertIb}). For that purpose note
that one can rewrite the right-hand side of (\ref{eq:SylvPertIb}) as
 \begin{align} \label{eq:SylvPertI-2b}
\what X_2^* \delta M  \wtd X_1 =   \what X_2^*
M^{1/2} M^{-1/2} \delta M  {\wtd M}^{-1/2}{\wtd M}^{1/2} \wtd X_1 {\wtd \Lambda_1}\,.
\end{align}
Recall, that from (\ref{PertStep2}) it follows that
$ \what Q_2^* \equiv \what X_2^* M^{1/2}$ and $\wtd Q_1 \equiv   {\wtd M}^{1/2} \wtd X_1$ have unitary columns, which together with (\ref{eq:SylvPertI-2b}) gives
 \begin{align*} 
 \what X_2^* \delta M  \wtd X_1  =\what Q_2^* M^{-1/2} \delta M  {\wtd M}^{-1/2} \wtd Q_1 \,.
\end{align*}

Applying  \cite[Lemma 2.3]{Li1999} to obtain the bounds on the solution of a structured Sylvester
  equation (see also \cite{Li1999a}), on (\ref{eq:SylvPertII}) one gets:
\begin{align}
\| \what X_2^* \delta M  \wtd X_1 \| \leq \frac{1}{\RelGap_p} \cdot  \|M^{-1/2} \delta M  {\wtd M}^{-1/2} \|
\label{eq:Bound1b}
\end{align}
where $\|\cdot\|$ stands for any unitary invariant norm, and  $\RelGap_p$ is defined as in  (\ref{relgap21b}). Now from (\ref{eq:Bound1b}) directly follows  bound (\ref{ocjsinpar2b}).
\end{proof}

\subsection{The main result}

As we have mentioned in our road-map, form (\ref{neq:NormSinTheta}) follows that
the upper bound for  $$\|\sin{\Theta_M({\cal X}_{1}, \wtd {\cal X}_{1} )} \|$$ will be obtained as the sum of the bounds for
$\|\sin{\Theta_M({\cal X}_{1}, \what {\cal X}_{1} )} \|$ and
$\|\sin_M{\Theta(\what {\cal X}_{1}, \wtd {\cal X}_{1} )} \|$.
Thus we have the following theorem:
\begin{tm} \label{Final_SinTheta_bound}
Let $(H, M)$ be a Hermitian pair and
let $(\wtd H, \wtd M)$ be the perturbed pair.
Let $X = \begin{bmatrix}X_{1} & X_{2} \end{bmatrix}$ and
$\wtd X = \begin{bmatrix}\wtd X_{1} & \wtd X_{2} \end{bmatrix}$, be
non-singular matrices  which simultaneously diagonalize the pairs
$(H, M)$ and $(\wtd H,  \wtd M)$, as in (\ref{diagvec1}) and (\ref{diagvec1t}), respectively.
If
\begin{align*}
\eta_M:= \|M^{-1/2} \delta M  M^{-1/2}\|_2 < \frac{1}{2}\,,
\end{align*}
and if (\ref{condonSpectA}) or (\ref{condonSpectB}) hold, then
\begin{align}\label{Final-ocjsinpar}
\|\sin{\Theta_M({\cal X}_{1}, \wtd {\cal X}_{1} )} \| & \leq
\frac{1}{\RelGap} \cdot \Psi_H +
\frac{1}{\RelGap_p} \cdot  \frac{\sqrt{1-\eta_M}}{\sqrt{1 - 2\, \eta_M}}\cdot\Psi_M \,,
\end{align}
where $\|\sin{\Theta_M({\cal X}_{1}, \wtd {\cal X}_{1} )} \|$ --- the sine of the angle between the subspaces in $M$ scalar product ---
is defined by (\ref{neq:NormSinTheta_pomB}),
and $\RelGap$ and $\RelGap_p$ are defined by
(\ref{relgap21}) and (\ref{relgap21b}), respectively.
\end{tm}
\begin{proof}
Using (\ref{neq:NormSinTheta}), (\ref{ocjsinpar2}), (\ref{def:SinTheta-M2b}) and
(\ref{ocjsinpar2b}) and the multiplicative properties of unitary invariant matrix norms one gets
\begin{align} \label{Final-ocjsinpar-midle-Step}
\|\sin{\Theta_M({\cal X}_{1}, \wtd {\cal X}_{1} )} \| & \leq
\frac{1}{\RelGap} \cdot \Psi_H  +
\frac{1}{\RelGap_p} \cdot  \Psi_M \cdot \|Y_{11}^{-1}\|_2 \,,
\end{align}
where $Y_{11}= \sqrt{I -  \wtd X_1^* \delta M  \wtd X_1}$ is defined as in (\ref{def:CholFact}).
It left us to compute the bound for $\|Y_{11}^{-1}\|_2$. Using the $\wtd M$-orthogonality of $\wtd X$ it
can be easily seen that  $\wtd X$ and $M^{-1/2} ( I + M^{-1/2} \delta M M^{-1/2})^{-1/2}$ are
unitarily similar, that is that exists unitary matrix $Q$ such that
\begin{align} \label{help1A}
\wtd X = M^{-1/2} ( I + M^{-1/2} \delta M M^{-1/2})^{-1/2} Q
\end{align}
Now we can proceed, note that
\begin{align} \label{help1B}
\|\left(I -  \wtd X_1^* \delta M  \wtd X_1\right)^{-1/2}\|_2 \leq
\frac{1}{\sqrt{1 -  \|\wtd X_1^* \delta M  \wtd X_1\|_2}} \leq
 \frac{1}{\sqrt{1 -  \|\wtd X^* \delta M  \wtd X\|_2}} \,.
\end{align}
Set $W= M^{-1/2} \delta M M^{-1/2}$, then from (\ref{help1A}) follows
\begin{align}\label{helpNormXt}
\|\wtd X^* \delta M  \wtd X\|_2 = \| ( I + W)^{-1/2} W  ( I + W)^{-1/2}\|_2
\leq \frac{\eta_M}{1-\eta_M}\,.
\end{align}
Finally inserting (\ref{helpNormXt}) in (\ref{help1B}) one gets
\begin{align} \label{help1C}
\|\big(I -  \wtd X_1^* \delta M  \wtd X_1\big)^{-1/2}\| \leq
\frac{\sqrt{1 - \eta_M}}{\sqrt{1 -  2\, \eta_M }}  \,.
\end{align}
Now, insert (\ref{help1C}) in (\ref{Final-ocjsinpar-midle-Step}) to get (\ref{Final-ocjsinpar}), which completes the proof.
\end{proof}
An alternative version---that is to say a version where alternative relative perturbation
sizes feature---can be obtained using Lemma \ref{lem1}.
\begin{ko}\label{kor_main}
Under the assumptions of Theorem \ref{Final_SinTheta_bound} we have the estimate
\begin{align}\label{Final-ocjsinpar2b}
\|\sin{\Theta_M({\cal X}_{1}, \wtd {\cal X}_{1} )} \| & \leq
\frac{1}{\RelGap} \cdot \frac{\Phi_H}{\sqrt{1-\eta_H}} +
\frac{1}{\RelGap_p} \cdot  \frac{\Phi_M}{\sqrt{1 - 2\, \eta_M}} \,,
\end{align}
where
$\Phi_M=\|M^{-1/2}\delta MM^{-1/2}\|$ and $\Phi_H=\|H^{-1/2}\delta H H^{-1/2}\|$.\end{ko}

\subsubsection{Weakening the assumption on the spectral dichotomy}
Note that theorem \ref{gptmsubspstep2} requires the special structure on specters of
$\wtd \Lambda_1$ and  $\what \Lambda_2$. The reason for this lies in the more involved analysis
of the structures Sylvester equation (\ref{eq:S2}), see the comment in the introduction to \cite{Li1999a}.

This limitation can be overcome by the use of the Frobenius norm instead of spectral norm. Thus,
the next theorem contains the perturbation bound similar to the one from Theorem
\ref{gptmsubspstep2} given for $\left\|\what X^*_2 M \wtd X_1 \right\|_F$, without any additional
assumptions on spectral configuration of the pair $(H,M)$.

\begin{tm} \label{gptmsubspstep2_Fro}
Let $(\wtd H, M)$, $(\wtd H, \wtd M)$, $\what X = \begin{bmatrix}\what X_{1} & \what X_{2} \end{bmatrix}$ and
$\wtd X = \begin{bmatrix} {\wtd X}_{1} & {\wtd X}_{2} \end{bmatrix}$, be as in Theorem \ref{gptmsubspstep2}.
Then
\begin{align}\label{ocjsinpar2b_Fro}
\left\|\what X^*_2 M \wtd X_1 \right\|_F & \leq
\frac{\Psi_M^{\|\cdot\|_F}}{\RelGap_{\rm comp}} \,,
\end{align}
where we remember the definition $\Psi_M^{\|\cdot\|_F}=\|M^{-1/2} \delta M  {\wtd M}^{-1/2} \|_F$
from (\ref{Psi_Fro}) and we assume that
\begin{eqnarray} \label{relgap21b_Fro}
\RelGap_{\rm comp} :=
\min_{\stackrel{\what \lambda_i \in \what \Lambda_2 }{ \wtd \lambda_j \in \wtd\Lambda_1}}
\frac{|\what \lambda_i - \wtd \lambda_j|}{\phantom{a} \wtd \lambda_j }~
\end{eqnarray}
is strictly larger than zero.
\end{tm}

\begin{proof} The first part of the proof is similar to the proof of theorem \ref{gptmsubspstep2} up to
the equality (\ref{eq:SylvPertI-2b}). Thus we continue the proof from there, that is one can write:
\begin{align} \label{eq:SylvPertIb_Fro}
\what \Lambda_2 \what X_2^* M  \wtd X_1 -  \what X_2^* M  \wtd X_1 \wtd \Lambda_1 =
 - \what X_2^* \delta M  \wtd X_1 \wtd \Lambda_1  \,,
\end{align}
and
 \begin{align} \label{eq:SylvPertII_Fro}
 \what X_2^* \delta M  \wtd X_1  =\what Q_2^* M^{-1/2} \delta M  {\wtd M}^{-1/2} \wtd Q_1\,,
\end{align}
where $ \what Q_2^* \equiv \what X_2^* M^{1/2}$ and $\wtd Q_1 \equiv   {\wtd M}^{1/2} \wtd X_1$ have unitary columns.

By interpreting (\ref{eq:SylvPertIb_Fro}) and (\ref{eq:SylvPertII_Fro}) component-wise if follows
\begin{align*}
(\what \Lambda_2)_{ii} (\what X_2^* M  \wtd X_1)_{ij} -  (\what X_2^* M  \wtd X_1)_{ij} (\wtd \Lambda_1)_{jj} =
 - (\what Q_2^* M^{-1/2} \delta M  {\wtd M}^{-1/2} \wtd Q_1)_{ij} (\wtd \Lambda_1)_{jj}  \,,
\end{align*}
or
\begin{align} \label{eq:SylvComponentI}
 (\what X_2^* M  \wtd X_1)_{ij}  =
 -\frac{(\wtd \Lambda_1)_{jj}}{(\what \Lambda_2)_{ii} - (\wtd \Lambda_1)_{jj}} \left(
 (\what Q_2)_{(:,i)}^* M^{-1/2} \delta M  {\wtd M}^{-1/2} (\wtd Q_1)_{(:,j)}
 \right)   \,,
\end{align}
where $(Q)_{(:,j)}$ denotes $j$-th column of the matrix $Q$.

By computing the Frobenius norm from (\ref{eq:SylvComponentI}) we have
\begin{align}\label{eq:kern}
 \|\what X_2^* M  \wtd X_1\|_F^2  = \sum\limits_{i=k+1}^n \sum\limits_{i=k+1}^n
 \frac{1}{\left| \frac{(\what \Lambda_2)_{ii} - (\wtd \Lambda_1)_{jj}}{(\wtd \Lambda_1)_{jj}}\right|^2} \left(
 (\what Q_2)_{(:,i)}^* M^{-1/2} \delta M  {\wtd M}^{-1/2} (\wtd Q_1)_{(:,j)}
 \right)^2   \,,
\end{align}
which gives
\begin{align} \label{eq:Bound1b_Fro}
\| \what  X_2^* M  \wtd X_1 \|_F \leq \frac{1}{\RelGap_{\rm comp}} \cdot  \|\what Q^* M^{-1/2} \delta M  {\wtd M}^{-1/2} \wtd Q_1 \|_F\,.
\end{align}
Now from (\ref{eq:Bound1b_Fro}), noting that $\what Q$ and $\wtd Q$ are both unitary, we obtain (\ref{ocjsinpar2b_Fro}).
\end{proof}

We can now give a Frobenius norm version of Theorem \ref{Final_SinTheta_bound}.

\begin{tm} \label{Final_Frob _bound}
Let $(H, M)$ be a Hermitian pair and
let $(\wtd H, \wtd M)$ be the perturbed pair.
Let $X = \begin{bmatrix}X_{1} & X_{2} \end{bmatrix}$ and
$\wtd X = \begin{bmatrix}\wtd X_{1} & \wtd X_{2} \end{bmatrix}$, be
non-singular matrices  which simultaneously diagonalize the pairs
$(H, M)$ and $(\wtd H,  \wtd M)$, as in (\ref{diagvec1}) and (\ref{diagvec1t}), respectively.
If the spectra are separated so that
$\RelGap_{\rm comp}>0$ then
\begin{align}\label{Final-ocjsinpar2}
\|\sin{\Theta_M({\cal X}_{1}, \wtd {\cal X}_{1} )} \|_{F} & \leq
\frac{1}{\RelGap} \cdot \Psi_H^{\|\cdot\|_{F}} +
\frac{1}{\RelGap_{\rm comp}} \cdot \Psi_M^{\|\cdot\|_{F}} \,,
\end{align}
where $\|\sin{\Theta_M({\cal X}_{1}, \wtd {\cal X}_{1} )} \|_{F}$ --- the sine of the angle between the subspaces in $M$ scalar product ---
is defined by (\ref{neq:NormSinTheta_pomB}),
and $\RelGap$ and $\RelGap_{\rm com}$ are defined by
(\ref{relgap21}) and (\ref{relgap21b_Fro}), respectively.
\end{tm}

\section{Numerical examples}\label{num}

It is not easy to numerically compare the eigenvector estimates for the perurbations of matrix pencils. The reason is that there is no canonical norm for the analysis of the eigenvector problem. For instance, assume that we have a positive definite symmetric pencil $(H,M)$, then any of the matrix
dependent norms (and the associated scalar products)
$$
\|\!|x\|\!|_{\alpha,\beta}=\sqrt{\alpha~ x^*Hx+\beta~ x^*Mx},\qquad \alpha\geq0, \beta\geq0, \text{ and } \alpha\beta\ne0\,
$$
is a meaningful candidate as well as is the standard Euclidean norm $\|\cdot\|$.

Applying any of the competing estimates in a situation for which they were not designed, is only possible after a nontrivial
intervention which often severely affects the sharpness of the result. To
this end we use the same set of problems for various approaches and compare them by comparing how well they
are doing a job they were designed for. More to the point, the estimate of the type $\textrm{Left}(\kappa)\leq\textrm{Right}(\kappa)$ --- where $\kappa\in\R$ is
a parameter --- is considered asymptotically sharp if
\begin{equation}\label{eq:exact_def}
\lim_{\kappa\to\infty}\frac{\textrm{Left}(\kappa)}{\textrm{Right}(\kappa)}=1.
\end{equation} Such property of
an estimator is sometimes called the asymptotic exactness of an estimator, see \cite{Grubivsi'c2009a}
and relation (\ref{eq:exact_def}) below.

\begin{rem}If we were to adopt the philosophy of \cite{Hetmaniuk2006a},
we would consider the family of eigenvalue problems $(H,\alpha H +\beta M)$ --- assuming $\alpha, \beta\in\R$ are such that
$\alpha H +\beta M$ is hermitian positive definite ---  and ask for such $\alpha$ and $\beta$ which are in some sense optimal. We cannot
 give an answer to the question of the choice the optimal energy norm now, but we might return to the question in future work.
Instead, we note that given $\alpha,\beta\in\R$
such that $\alpha H +\beta M$ and $H$ are Hermitian positive definite reduces the problem to the one we can handle. The eigenvalues $\lambda_i$ of the
matrix pair $(H,M)$ and $\lambda^{\alpha,\beta}_i$ of the pair $(H,\alpha H+\beta M)$ are related by the transformation $\mu_i=\lambda_i/(\alpha\lambda_i+\beta_i)$.
\end{rem}

\subsection{Perturbations of eigenspaces in the energy norm}\label{Acc1}

We will now use the theory form the preceeding section to study the rotation of eigenvectors of a parameter dependent family of eigenvalue problems
$$
H_\kappa=H_b+\kappa H_e,\qquad \kappa\gg1.
$$
Here $H_b$ is positive definite, $H_e$ is a positive semi-definite matrix and we are interested
in the estimate of the rotation of eigenvectors in the changing energy norm
$$
\|\!|x\|\!|_{H_\kappa}=\sqrt{x^*H_\kappa x}.
$$
For some further motivation for studying these problems see the Appendix

To this end we note that eigenvector problems
\begin{align}
H_\kappa v=\lambda v,&\;\;v=\frac{1}{\lambda}H_\kappa v,\;\;v=\lambda H_\kappa^{-1}v\\
H_\kappa^{-1}v=\frac{1}{\lambda}v,&\;\;H_\kappa^{-1}v=\frac{1}{\lambda^2}H_\kappa v
\end{align}
have the same eigenvectors. Furthermore, it is known, \cite{Grubivsi'c2009,Warburton2006} that as $\kappa$ tends to infinity the eigenvalues of $H_\kappa$ either tend to infinity or, they converge to the nonzero eigenvalues
of $$L_b:=P_{\Ker(H_e)}H_\kappa\Big|_{\Ker(H_e)}.$$

Subsequently, we decompose the space $\R^n=\Ker(H_e)\oplus(\Ker(H_e))^\perp$ and, without
reducing the level of generality---see \cite[Formula (12)]{Warburton2006}---think of $H_\kappa$ as the block operator matrix
\begin{equation}\label{eq:penal}
H_\kappa=\begin{bmatrix}
L_b& R_b^*\\R_b&W_b
\end{bmatrix}+\kappa\begin{bmatrix}0&0\\0&H_e\end{bmatrix}.
\end{equation}
We also denote the block diagonal of $H_\kappa$ with
$$
D_\kappa=\begin{bmatrix}L_b&\\& W_b +\kappa H_e\end{bmatrix}
$$
and compute
\begin{align*}
\|D_\kappa^{-1/2}(D_\kappa-H_\kappa)D_\kappa^{-1/2}\|&=\Big\|\begin{bmatrix}
0&L_b^{-1/2}R_b^*(W_b+\kappa E_B)^{-1/2}\\
(W_b+\kappa E_B)^{-1/2}R_bL_b^{-1/2}&0
\end{bmatrix}\Big\|\\&=\frac{1}{\sqrt{\kappa}}\Big\|\begin{bmatrix}
0&L_b^{-1/2}R_b^*(\frac{1}{\kappa}W_b+E_B)^{-1/2}\\
(\frac{1}{\kappa}W_b+E_B)^{-1/2}R_bL_b^{-1/2}&0
\end{bmatrix}\Big\|\\
&=O\big(\frac{1}{\sqrt{\kappa}}\big).
\end{align*}
Let us introduce the perturbation estimate $\eta_{H_\kappa}:=\|D_\kappa^{-1/2}(D_\kappa-H_\kappa)D_\kappa^{-1/2}\|$. With this we note the following
inequalities
\begin{align}
|x^*H_\kappa x-x^*D_\kappa x|&\leq\eta_{H_\kappa} x^* D_\kappa x\\
|x^*H_\kappa^{-1} x- x^*D_\kappa^{-1} x|&\leq\frac{\eta_{H_\kappa}}{1-\eta_{H_\kappa}}~x^*D_\kappa^{-1} x.\label{cf_eq}
\end{align}
Obviously, with this analysis we can chose
\begin{equation}\label{eqInv}
\eta_{H_\kappa^{-1}}=\frac{\eta_{H_\kappa}}{1-\eta_{H_\kappa}}
\end{equation}
and so we can apply Theorem \ref{Final_SinTheta_bound} directly.
\begin{rem}
This discussion indicates that it is easy, within this theory, to switch the roles of
$H$ and its inverse $H^{-1}$. This is so because an estimate on the perturbation of the one implies the relative estimate for the
perturbation of the other. A similar feature is shared by the relative gap from (\ref{Final-ocjsinpar2}) since
$$
\frac{|\frac{1}{\lambda}-\frac{1}{\mu}|}{\sqrt{\frac{1}\lambda\frac{1}{\mu}}}=\frac{|\lambda-\mu|}{\sqrt{\lambda\mu}}.
$$
It is pleasing and useful --- when switching the roles of $H$ and $M$ --- that both ingredients of an estimate like
(\ref{Final-ocjsinpar2})  are robust with respect to inversion of the eigenvalues.
\end{rem}

For first simple experiments we consider the family of problems
\begin{equation}\label{eq:prim}
H_\kappa=\begin{bmatrix}2&-1&0\\-1&2&-1\\0&-1&2+\kappa\end{bmatrix},\quad
\mathbb{H}_\kappa=\begin{bmatrix}2&-1&0&0\\-1&2&-1&0\\0&-1&2&-1\\0&0&-1&2+\kappa\end{bmatrix}\qquad\kappa\gg1.
\end{equation}
In the first experiment we will see how do the ingredients of the estimates---relative gap and the
residual---feature in their performance.

By $\lambda_1^{H_\kappa}<\lambda_2^{H_\kappa}<\lambda_3^{H_\kappa}$
we denote the eigenvalues of $H_\kappa$ and by
$\lambda_1^{\mathbb{H}_\kappa}<\lambda_2^{\mathbb{H}_\kappa}<\lambda_3^{\mathbb{H}_\kappa}<\lambda_4^{\mathbb{H}_\kappa}$
the eigenvalues of $\mathbb{H}_\kappa$. We also use for eigenvectors the following notation
\begin{align*}
H_\kappa v_i^{H_\kappa}&=\lambda_i^{H_\kappa} v_i^{H_\kappa},\qquad i=1, 2, 3,\\
\mathbb{H}_\kappa v_i^{\mathbb{H}_\kappa}&=\lambda_i^{\mathbb{H}_\kappa} v_i^{\mathbb{H}_\kappa},\qquad i=1, 2, 3,4~.
\end{align*}

The behavior of the family of problems (\ref{eq:prim}) has been analyzed in \cite{Warburton2006} with the
help of the Gerschgorin theorem. Let us consider the eigenspace which  belongs to the eigenvalues
$\lambda_1^{H_\kappa}<\lambda_2^{H_\kappa}$ and $\lambda_1^{\mathbb{H}_\kappa}<\lambda_2^{\mathbb{H}_\kappa}$.
to this end we write the implicit partial diagonalization of $H_\kappa$ and $\mathbb{H}_\kappa$ in the
generic block matrix form
$$
\begin{bmatrix}
L_{b}&R_b^*\\
R_b&W_b+\kappa H_e
\end{bmatrix}\begin{bmatrix}V_\kappa\\ \hat W_\kappa\end{bmatrix}=\begin{bmatrix}V_\kappa\\\hat W_\kappa\end{bmatrix}
\Lambda_\kappa
$$
where $L_b$, $W_b$, $R_b$ and $H_e$ are as in (\ref{eq:penal}) and $\Lambda_\kappa$ is
the diagonal matrix containing the targeted eigenvalues. The orthogonality property
$V_\kappa^*V_\kappa+\hat W_\kappa^*\hat W_\kappa=I$ together with the Gerschgorin theorem implies,
see \cite[pg. 3209]{Warburton2006}, the estimates
\begin{equation}\label{eq:Warburton}
\|L_bV_\kappa-V_\kappa\Lambda_\kappa\|=O\big(\frac{1}{\kappa}\big), \;\;\|V_\kappa^*V_\kappa-I\|=O\big(\frac{1}{\kappa^2}\big),\quad \|\hat W_\kappa\|=O\big(\frac{1}{\kappa}\big).
\end{equation}
In the example that follows we show this explicitly on the model problem and indicate a possible
dependence on $\kappa$ of the otherwise unaccessible matrix $V_\kappa$.
\begin{ex}\label{ex1}In this example we show that the estimates are \textit{asymptotically sharp}
---for the definition of this notion see (\ref{eq:exact_def}) below and reference \cite{Grubivsi'c2009a}
for a discussion of its significance in finite element computations---for the matrix $H_\kappa$.
 For this problem we have for eigenvalues and eigenvectors
 \begin{align*}
\lambda_1^\kappa&=1-\frac{1}{2 \kappa }+\frac{3}{8 \kappa ^2}-\frac{55}{128 \kappa ^4}+\frac{1}{2 \kappa ^5}+O\left(\frac{1}{\kappa^6}\right)\\
\lambda_2^\kappa&=3-\frac{1}{2 \kappa }-\frac{3}{8 \kappa ^2}+\frac{55}{128 \kappa ^4}+\frac{1}{2 \kappa
   ^5}+O\left(\frac{1}{\kappa^6}\right)\\
\lambda_3^\kappa&=\kappa +2+\frac{1}{\kappa }-\frac{1}{\kappa^5}+O\left(\frac{1}{\kappa^6}\right)
\end{align*}
\begin{align*}
v_1^\kappa&=\begin{bmatrix}
1 +\frac{1}{2\kappa}+\frac{5}{8 \kappa^2 }-\frac{1}{2 \kappa^3}+\frac{7}{128 \kappa^4}+\frac{1}{2 \kappa^5}-\frac{675}{1024
\kappa^6}+O\left(\frac{1}{\kappa^7 }\right)\\
1 +\frac{1}{\kappa}+\frac{1}{2 \kappa^2 }-\frac{3}{8 \kappa^3}+\frac{55}{128 \kappa^5}-\frac{1}
{2 \kappa^6}+O\left(\frac{1}{\kappa^7 }\right)\\\frac{1}{\kappa}\end{bmatrix},\\
v_2^\kappa&=\begin{bmatrix}
-1 +\frac{1}{2\kappa}-\frac{5}{8 \kappa^2 }-\frac{1}{2\kappa^3}-\frac{7}{128 \kappa^4}+\frac{1}{2 \kappa^5}+\frac{675}{1024
   \kappa^6}+O\left(\frac{1}{\kappa^7 }\right)\\
   1 -\frac{1}{\kappa}+\frac{1}{2\kappa^2}+\frac{3}{8 \kappa^3}-\frac{55}{128 \kappa^5}-
   \frac{1}{2\kappa^6}+O\left(\frac{1}{\kappa^7 }\right)\\\frac{1}{\kappa}
   \end{bmatrix},\\
v_3^\kappa&= \begin{bmatrix}
\left(\frac{1}{\kappa }\right)^2-\left(\frac{1}{\kappa }\right)^4+O\left(\frac{1}{\kappa^6 }\right)\\-\frac{1}{\kappa }+\left(\frac{1}{\kappa }\right)^5+O\left(\frac{1}{\kappa^6 }\right)\\1
\end{bmatrix}
\end{align*}
and $\eta_{H_\kappa}=\sqrt{\frac{2}{6+3\kappa}}$. Note that the
matrix $\begin{bmatrix}V_\kappa&\hat W_\kappa\end{bmatrix}^*$ has
columns given by $v_1^\kappa$ and $v_2^\kappa$, and so we can see the dependence $V_\kappa$
on the penalty parameter in this example explicitly.
\begin{figure}[ht]
\includegraphics[height=8cm,width=12cm]{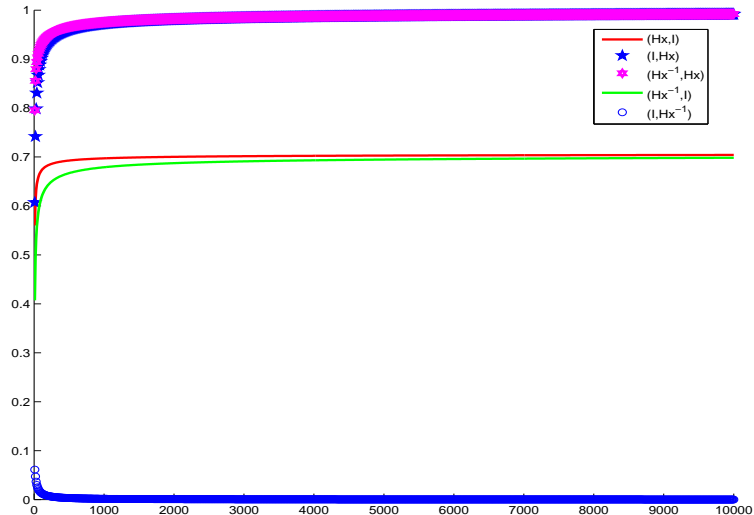}
\caption{Numerical experiment for Example \ref{ex1}. The experiment
demonstrates the notion of the \textit{asymptotic sharpness}. In this plot we have depicted the
effectivity quotient against the penalty parameter, see Example \ref{ex1} for the definition.}\label{Fig2}
\end{figure}
Using (\ref{eqInv}) we obtain
$$
\textrm{Right}_\kappa:=\frac{1}{\RelGap} \cdot \frac{\eta_{H_\kappa^{-1}}}{\sqrt{1 - \eta_{H_\kappa^{-1}}}}  +
\frac{1}{\RelGap_p} \cdot  \frac{\eta_{H_\kappa}}{\sqrt{1 - 2\, \eta_{H_\kappa}}}=O(\frac{1}{\sqrt{\kappa}}).
$$
On the other hand, a simple computation and Theorem \ref{Final_SinTheta_bound} yield, cf. (\ref{eq:Warburton}) that
$$
\textrm{Left}_\kappa:=\sin\Theta_{H_\kappa}(\Ran[v_1^\kappa \;v_2^\kappa],~\Ran[v_1^\infty \;v_2^\infty])=O(\frac{1}{\sqrt{\kappa}}).
$$
Here we have used the symbol $v_i^\infty$, $i=1,2,3$ to denote the limit eigenvectors of
$v_i^\kappa$, $i=1,2,3$ as $\kappa\to\infty$. They are also
the eigenvectors of the limit matrix
$$
H_\infty=\begin{bmatrix}2&-1&0\\-1&2&0\\0&0&0\end{bmatrix}.
$$
This shows that the energy norm estimate is sharp when viewed as the function of $\kappa$. On the other
hand a simple computation reveals that any of the $\sin\Theta$ theorems from \cite{Davis1970,Grubivsi'c2007,Li1999} yields a similar $O(\frac{1}{\sqrt{\kappa}})$---or even worse\footnote{The
residual estimate (\ref{eq:Warburton}) gets spoilt when we chose the orthonormal basis
for $\Ran[v_1^\infty \;v_2^\infty]$ as the columns of $V_\kappa$ are not orthonormal.}
for the $O(1)$---upper estimate for the
$$
\sin\Theta(\Ran[v_1^\kappa \;v_2^\kappa],~\Ran[v_1^\infty \;v_2^\infty])=O(\frac{1}{\kappa}).
$$

We now turn our attention to the study of the asymptotic sharpness --- in the sense of (\ref{eq:exact_def}) --- of our estimates on concrete examples This can be proved by direct computation for the case of our estimate applied to the
matrix pairs $(H_\kappa^{-1}, H_\kappa)$ and $(I,H_\kappa)$, cf. Example \ref{ex42} for further discussion.
This shows that a notion of sharpness---a $\sin\Theta$ theorem is considered to be sharp if there is
a perturbation in the allowed class of perturbations such that the bound is attained---for the estimates
of the rotation of eigenvectors is a delicate question. Let us note that we will call
$\frac{\textrm{Left}_\kappa}{\textrm{Right}_\kappa}$
the \textit{effectivety quotient}.
\end{ex}
\begin{ex}\label{ex42}
In this example we perform a Matlab experiment in which we evaluate
the estimate of Corollary \ref{kor_main} for the matrix pairs
$$
(\mathbb{H}_\kappa,I),\;(I,\mathbb{H}_\kappa),\;(\mathbb{H}_\kappa^{-1},\mathbb{H}_\kappa),\;(\mathbb{H}_\kappa^{-1},I),\;(I,\mathbb{H}_\kappa^{-1}).
$$
The results are presented on Figure \ref{Fig2b}. The results further illustrate the delicacy
of the issue of the sharpness of $\sin\Theta$ theorems.
\begin{figure}[ht]
\includegraphics[height=8cm,width=12cm]{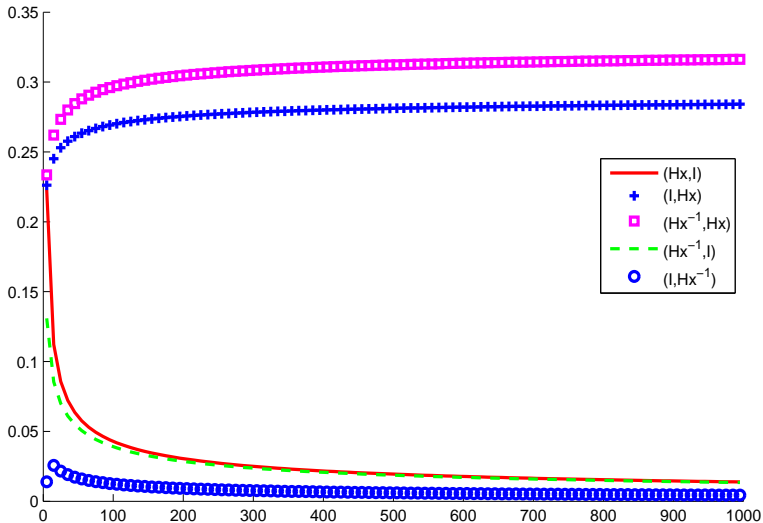}
\caption{Numerical experiment for Example \ref{ex42}. In this plot we have depicted the
effectivity quotient against the penalty parameter, see Example \ref{ex1} for the definition.}\label{Fig2b}
\end{figure}
Namely, the estimates are not asymptotically sharp for any of the considered matrix pairs, but the energy norm
estimates---that is estimates for the pairs $(I,\mathbb{H}_\kappa)$ and $(\mathbb{H}_\kappa^{-1},\mathbb{H}_\kappa)$---
are of the same order of the magnitude as the error---this can be seen from the fact that the
effectivity quotients converge to a constant---where es in the case of the estimates
for the other norms the effecivity quotients converge to zero. These convergence claims can be verified by a
direct symbolic computation.
This example shows that both the choice of a measure of the spectral gap as well as the choice of the measure of the residual play
a role in obtaining high performance estimators, since it was the influence of the measure of the relative
gap which guaranteed the asymptotic sharpness in Example \ref{ex1}, compare Figures \ref{Fig2} and \ref{Fig2b}.
\end{ex}

\subsection{A Matrix Market example}\label{Acc2}
For a further illustration of an effect similar to mass lumping we will consider the generalized eigenvalue problem
\begin{eqnarray*} 
H x & = & \lambda M x ,
\end{eqnarray*}
where the matrix $H$ is taken from the Matrix Market basis, see \cite{MatrixMar}. We choose
$H$ from the set CYLSHELL: Finite element analysis of cylindrical shells matrices. From this  test set
we took the matrix \verb+s1rmq4m1.mtx+ which is  real symmetric positive definite, $5489 \times 5489$
matrix with $143300$ entries. This matrix is obtained by finite element discretization of an octant of a cylindrical shell. The ends of the cylinder are free.


For the matrix $M$ we took diagonal matrix with---in Matlab notation---$\texttt{disg}(1:n)$ and we consider random perturbations $\delta H$ and $\delta M$, which satisfy
\begin{eqnarray*} 
|(\delta H)_{ij}| \leq \eta_H  |H_{ij}|, \qquad |(\delta M)_{ij}| \leq \eta_M  |M_{ij}|,
\end{eqnarray*}
where $\eta_H = \eta_M = 10^{-8}$. The above assumption means that zeros remain unperturbed and
we have chosen the $M$ matrix whose norm explodes as $n\to\infty$. This is a reasonable choice
for our method, since the technique of our proof can readily be adapted to yield the same result
for some unbounded pair of operators in a Hilbert space.

As a comparison we consider one of the well known the standard perturbation bound for matrix pairs is given by the theorem of Stewart and Sun from \cite[Chapter VI]{Stewart1990}. To this end, let $(H, M)$ be a symmetric definite pair, such that (\ref{diagvec1}) holds. That is,
let $X = \lmat{cc} X_1 & X_2 \rmat$ be such that
\begin{eqnarray} \label{simdiag1}
\begin{bmatrix} X_1^* \\ X_2^* \end{bmatrix} H \begin{bmatrix} X_1 & X_2 \end{bmatrix} = \begin{bmatrix}\Lambda_1 & \\  &
\Lambda_2 \end{bmatrix} \quad
\begin{bmatrix}X_1^* \\ X_2^* \end{bmatrix} M \begin{bmatrix}X_1 & X_2 \end{bmatrix} = \begin{bmatrix} I_k & \\
 & I_{n-k} \end{bmatrix},
 \end{eqnarray}
where
\begin{eqnarray*}
\Lambda_1 = \diag(\lambda_1, \ldots, \lambda_k), & & \Lambda_2 =
\diag(\lambda_{k+1}, \ldots, \lambda_n),
 \end{eqnarray*}
and $X_1 \in \C^{n \times k}$, $X_2 \in \C^{n \times
{n-k}}$.
 The following
theorem contains a bound for the Frobenius norm of the diagonal
matrix which contains the sines of the canonical angles between
eigenspace $\Ran(X_1)$ and corresponding perturbed eigenspace ${\cal
R}(\tilde X_1)$.

\begin{tm}[Sun]\label{korssn}
Let the definite pair $(H, M)$ be decomposed as in
(\ref{simdiag1}) where $X_1$ and $X_2$ have orthonormal columns.
Let the analogous decomposition be given for the pair $(\tilde H,
\tilde M) \equiv ( H+\delta H, M + \delta M)$. If
$$
\delta = \min \left\{ \frac{|\tilde \lambda - \lambda|}{\sqrt{1 + \tilde
\lambda^2} \sqrt{1 + \lambda^2}}\,;  \lambda \in \varrho(\Lambda_1),
\tilde \lambda \in \varrho(\tilde \Lambda_2)\right\},
$$ then
\begin{eqnarray}\label{stsunsin1}
\| \sin{\Theta[\Ran(X_1), \Ran(\tilde X_1)]} \|_F \leq
\frac{\sqrt{\| H^2 + M^2\|}}{\gamma(H,M) \gamma(\tilde H, \tilde
M)} \frac{\sqrt{\| \delta H X_1 \|_F^2 + \| \delta M X_1
\|_F^2}}{\delta}\,,
 \end{eqnarray}
where
\begin{equation}\label{defpair}
\gamma(H,M) = \min_{\stackrel{ x \in \C^n}{\| x \| =1 }}
| x^* (H + \imath M ) x| = \min_{\stackrel{ x \in \C^n}{\| x \| =1 }}
\sqrt{ (x^* H x)^2 + (x^* M x)^2} > 0.
\end{equation}
\end{tm}

We estimate the perturbation of invariant subspace which corresponds with first four smallest eigenvalues of the matrix pair $(H,M)$. The experiment is to be understood in the context of the testing
of the asymptotic sharpness of the estimator as in the definition (\ref{eq:exact_def}).
\begin{ex}[The performance of our estimate]
The exact perturbation gives:
\begin{align*}
\|\sin{\Theta_M({\cal X}_{1}, \wtd {\cal X}_{1} )} \|  \approx 6.727 \cdot 10^{-7},
\end{align*}
while our bound (\ref{Final-ocjsinpar}) gives
\begin{align*}
\|\sin{\Theta_M({\cal X}_{1}, \wtd {\cal X}_{1} )} \|  \leq 8.6721 \cdot 10^{-4}\,.
\end{align*}
\end{ex}
\begin{ex}[The performance of the Stewart-Sun bound]
The bound (\ref{stsunsin1}) here is not satisfactory due the fact that
$\gamma(H,M) = 1$, and $\gamma(H+\delta H, M + \delta M) \approx 1 + \varepsilon$. On the other hand
the gap $\delta \sim  10^{-6}$ and $\sqrt{\| H^2 + M^2\|}\sim 10^{5}$. Together with
\[
\sqrt{\| \delta H X_1 \|_F^2 + \| \delta M X_1 \|_F^2} = 3.872 \cdot 10^{-6}\,,
\]
we have
\begin{align*}
\| \sin{\Theta[\Ran(X_1), \Ran(\tilde X_1)]} \|_F   \leq  6 \cdot 10^{5}\,.
\end{align*}
\end{ex}

\section*{Acknowledgement}
Luka Grubi\v{s}i\'{c} was supported by the grant:
``Spectral decompositions -- numerical methods and applications'', Grant Nr. 037-0372783-2750 of the Croatian MZOS,
Ninoslav Truhar was supported by the grant:
``Passive control of mechanical models '', Grant Nr. 235-2352818-1042 of the Croatian MZOS.

\bibliographystyle{abbrv}
\def\cprime{$'$}

\appendix
\section{A motivation to study the problems of the large coupling limit}\label{App}
Consider positive definite eigenvector problems of the following type: find
$\psi$, $\|\psi\|=1$ and
$\lambda\in\R$ such that
\begin{equation}\label{eq:motiv}
H_\kappa\psi=H_b\psi+\kappa H_e\psi=\lambda\psi,
\end{equation}
where $H_b$ is positive definite matrix and $H_e$ is a semidefinite
perturbation which has a significant null space and $\kappa\gg 1$.
The presence of a large coupling constant \(\kappa\)
the singular perturbation $H_e$ causes the appearance of spurious, that is
nonphysical,
eigenvalues due to the non-zero component of $H_e$.
  It is our aim to obtain bounds on the rotation of eigenspaces
which is caused by this perturbation.

When considering the families of matrices/operators like $
H_\kappa=H_b+\kappa H_e,\qquad\kappa\gg1.
$ the parameter $\kappa$ is called the coupling ---
or depending on the context the penalty --- parameter.
The family of perturbations $\kappa H_e$
splits the spectrum of $H_\kappa$ into a bounded and an unbounded component as
$\kappa\to\infty$.

One typical example of a problem in this setting are the penalty methods for
Maxwell or Stokes' eigenvalue problems. For more information and further
references see \cite{Warburton2006}.
There, the authors analyze the dependence of the spectrum of the discretization
matrix of the
Maxwell's eigenvalue problem on the large coupling parameter and show---by a
very elegant Gerschgorin
type argument---that as $\kappa\to\infty$ the eigenvalues of interest converge
with the rate
proportional to $\kappa^{-1}$.

Let us note that the models where one considers the limits of the large penalty
are representative for a larger class of parameter dependent singularly
perturbed eigenvalue problems.
These problems typically appear in the study of optical
nano-devices, hard core scattering theory and in the analysis of
lower dimensional approximations to the 3D elasticity (like Arches and
Plates), see \cite{BenAmor2008,Brasche2005a,Demuth2000,Grubivsi'c2009}.
Another example is the so called ``lumped mass approximation'' in which
an auxiliary diagonal mass matrix \(\tilde M\) is constructed which generates
an equivalent scalar product. Such matrices are typically constructed by using quadrature
formulae and pseudo $L^2$ projections, see \cite{Banerjee1990}.
The analysis from \cite{BenAmor2008,Grubivsi'c2009}
shows that the eigenvalue estimates form \cite{Warburton2006} are sharp when
viewed in terms of
the dependence on the coupling constant, cf. Example \ref{ex1}.
\footnote{Explicit constants and their physical
interpretations are explicitly given in \cite{BenAmor2008,Grubivsi'c2009}.}

\end{document}